\documentclass[12pt]{amsart}

\usepackage{amsxtra,amssymb,amsthm,amsmath,amscd,url,listings}
\usepackage[utf8]{inputenc}
\usepackage{eucal}
\usepackage{fullpage}
\usepackage{scrtime}
\usepackage{hyperref}

\renewcommand{\leq}{\leqslant}
\renewcommand{\geq}{\geqslant}

\numberwithin{equation}{section}

\newcommand{\Cc}{\mathbf{C}}

\newcommand{\Aa}{\mathbf{A}}

\newcommand{\Zz}{\mathbf{Z}}
\newcommand{\Pp}{\mathbf{P}}

\newcommand{\Gg}{\mathbf{G}}

\newcommand{\Qq}{\mathbf{Q}}
\newcommand{\Fp}{\mathbf{F}}

\newcommand{\expect}{\mathbf{E}}

\newcommand{\mods}[1]{\,(\mathrm{mod}\,{#1})}

\newcommand{\frtr}[3]{t_{{{#1},{#2}}}({#3})}
\newcommand{\frfn}[2]{t_{{{#1},{#2}}}}

\newcommand{\ra}{\rightarrow}
\newcommand{\lra}{\longrightarrow}

\newcommand{\injecte}{\hookrightarrow}

\DeclareMathOperator{\rank}{rank}

\DeclareMathOperator{\frob}{\mathrm{Fr}}

\DeclareMathOperator{\Ind}{Ind}

\DeclareMathOperator{\Tr}{Tr}

\DeclareMathOperator{\swan}{Swan}

\DeclareMathOperator{\ft}{FT}
\DeclareMathOperator{\cond}{c}
\DeclareMathOperator{\sing}{Sing}
\DeclareMathOperator{\dual}{D}

\renewcommand{\rho}{\varrho}

\DeclareMathOperator{\PGL}{PGL}

\newcommand{\demi}{{\textstyle{\frac{1}{2}}}}

\newcommand{\sheaf}[1]{\mathcal{{#1}}}

\DeclareMathSymbol{\gena}{\mathord}{letters}{"3C}
\DeclareMathSymbol{\genb}{\mathord}{letters}{"3E}

\theoremstyle{plain}
\newtheorem{theorem}{Theorem}[section]
\newtheorem{lemma}[theorem]{Lemma}
\newtheorem{corollary}[theorem]{Corollary}

\newtheorem{proposition}[theorem]{Proposition}

\theoremstyle{remark}

\theoremstyle{definition}

\newtheorem{definition}[theorem]{Definition}

\newtheorem{example}[theorem]{Example}
\newtheorem{remark}[theorem]{Remark}

\renewcommand{\geq}{\geqslant}
\renewcommand{\leq}{\leqslant}

\begin{document}

\title{An inverse theorem for Gowers norms of trace functions over
  $\Fp_p$}
 
\author{\'Etienne Fouvry}
\address{Universit\'e Paris Sud, Laboratoire de Math\'ematique\\
  Campus d'Orsay\\ 91405 Orsay Cedex\\France}
\email{etienne.fouvry@math.u-psud.fr} 

\author{Emmanuel  Kowalski}
\address{ETH Z\"urich -- D-MATH\\
  R\"amistrasse 101\\
  8092 Z\"urich\\
  Switzerland} 
\email{kowalski@math.ethz.ch}

\author{Philippe Michel}
\address{EPFL/SB/IMB/TAN, Station 8, CH-1015 Lausanne, Switzerland }
\email{philippe.michel@epfl.ch}

\date{\today,\ \thistime}



\begin{abstract}
  We study the Gowers uniformity norms of functions over $\Zz/p\Zz$
  which are trace functions of $\ell$-adic sheaves. On the one hand, we
  establish a strong inverse theorem for these
  functions, and on the other hand this gives many explicit examples
  of functions with Gowers norms of size comparable to that of
  ``random'' functions.
\end{abstract}

\maketitle

\section{Introduction}

The \emph{Gowers uniformity norms} were introduced by Gowers in his
work on Szemer\'edi's theorem. As one sees from the definition
(see~\cite[Def. 11.2]{tv}), these norms (or a suitable power of them)
have very algebraic definitions when applied to functions defined over
a finite abelian group. In particular, one may consider a finite field
$k$, and attempt to understand the Gowers norm of functions of
algebraic nature on $k$. The most natural definition of such functions
seems to be the \emph{trace functions} of suitable sheaves, as we will
recall below. Indeed, in recent works~\cite{fkm, fkm-counting,
  fkm-primes}, we have shown that such functions (in the special case
$k=\Fp_p$) can be exploited powerfully in analytic arguments of
various types (amplification method for averages against Fourier
coefficients of modular forms, bilinear forms for averages over
primes, etc).
\par
In this note, we consider the Gowers norms of trace functions. Maybe
the most crucial issue in the study of these norms has been to
understand which bounded functions have ``large'' Gowers norm, a
suitable structural answer being known as an ``inverse theorem'' for
these norms. As it turns out, the rigidity of the structure of trace
functions, and especially Deligne's proof of the Riemann Hypothesis,
also leads to a rather precise structure theorem for Gowers norms of
trace functions over $\Fp_p$.
\par
Although this was not anticipated at first,\footnote{\ We thank
  B. Green for pointing out that this application of our result is of
  interest in combinatorics.} it also turns out that the estimates we
obtain give many simple explicit examples of functions with Gowers
norms of size comparable to that of ``random'' functions, in a precise
sense recalled below.  Since this fact may be of interest to people
interested in pseudorandomness measures of various sequences (see
also, among others, the papers~\cite{niederreite-rivat} of
Niederreiter and Rivat,~\cite{liu} of Liu and~\cite{fmrs} of Fouvry,
Michel, Rivat and S\'ark\"ozy), we first state a concrete result which
does not require any advanced algebraic-geometry language. In the
statement, $\|\cdot\|_{U_d}$ is the $d$-th Gowers uniformity norm
(normalized as in~\cite[Def. 11.2]{tv}), the definition of which is
recalled in Section~\ref{sec-prelims}.

\begin{theorem}\label{th-explicit}
For an odd prime $p$, and $x\in\Zz/p\Zz$, let
\begin{align}
  \varphi_1(x)&=\Bigl(\frac{f(x)}{p}\Bigr)\\
  \varphi_2(x)&=e\Bigl(\frac{\bar{x}}{p}\Bigr)\text{ if $x\not=0$ and
  }
  \varphi_2(0)=0,\\
  \varphi_3(x)&=\frac{S(x,1;p)}{\sqrt{p}} \text{ if $x\not=0$ and }
  \varphi_3(0)=-\frac{1}{\sqrt{p}},\\
  \varphi_4(x)&=\frac{1}{\sqrt{p}}\Bigl(p-|\{(u,v)\in\Fp_p^2\,\mid\,
  v^2=u(u-1)(u-x)\}|\Bigr)\text{ if $x\notin \{0,1\}$ and }\\
  \varphi_4(0)&=\varphi_4(1)=\frac{1}{\sqrt{p}},\nonumber
\end{align}
where $f\in\Zz[X]$ has degree $m\geq 1$ and is not proportional to the
square of another polynomial, $\bigl(\tfrac{\cdot}{p}\bigr)$ is the
Legendre symbol, and $S(a,b;c)$ is a classical Kloosterman sum. Then
for $d\geq 1$, we have
\begin{align*}
  \|\varphi_1\|_{U_d}^{2^d}&\leq (5m+10)^{(d+1)2^d}p^{-1},
  \\
  \|\varphi_2\|_{U_d}^{2^d}&\leq 15^{(d+1)2^d}p^{-1},
  \\
  \|\varphi_3\|_{U_d}^{2^d}&\leq 20^{(d+1)2^d}p^{-1},
  \\
  \|\varphi_4\|_{U_d}^{2^d}&\leq 25^{(d+1)2^d}p^{-1}.
\end{align*}
\end{theorem}

\begin{remark}
  In~\cite[Ex. 11.1.17]{tv}, Tao and Vu note that if $\varphi$ is a
  \emph{random function}, in the sense that the values $\varphi(x)$,
  for $x\in\Fp_p$, are independent random variables (on some
  probability space) with $|\varphi(x)|\leq 1$ for all $x$ and with
  expectation zero, then we have
$$
\expect(\|\varphi\|_{U_d}^{2^d})\ll p^{-1},
$$
where the implied constant depends only on $d$. Thus this result gives
concrete examples of functions which are as uniform as random
functions (note that, by Weil's bound for Kloosterman sums and by
Hasse's bound for the number of points on elliptic curves over finite
fields, we have $|\varphi_3|$, $|\varphi_4|\leq 2$). We are not aware of previous
examples with this property in the literature (though the cases of
$\varphi_1$ and $\varphi_2$ are accessible to techniques based only on the Weil
bounds for character sums.)
\end{remark}

We will explain the proof of this result in
Section~\ref{sec-refined}. We now discuss the inverse theorems for
general trace functions. We first recall the general setup of these
functions (further examples are given in
Example~\ref{ex-simple-examples} and in Section~\ref{sec-examples}).
\par
We fix a prime $p$ and a finite field $k$ of characteristic $p$. Let
$\ell\not=p$ be a prime number. For any algebraic variety $X/k$, any
finite extension $k'/k$ and $x\in X(k')$, we denote by
$\frtr{\sheaf{F}}{k'}{x}$ the value at $x$ of the trace function of
some $\ell$-adic (constructible) sheaf $\sheaf{F}$ on $X/k$. We will
write $\frfn{\sheaf{F}}{k'}$ for the function $x\mapsto
\frtr{\sheaf{F}}{k'}{x}$ defined on $X(k')$.
\par
We will always assume that some isomorphism $\iota\,:\,
\bar{\Qq}_{\ell}\lra \Cc$ has been chosen and we will allow ourselves
to use it as an identification. Thus, for instance, by
$|\frtr{\sheaf{F}}{k}{x}|^2$, we will mean
$|\iota(\frtr{\sheaf{F}}{k}{x})|^2$.
\par
Given any finite field $k$ and any function $\varphi\,:\, k\lra
\bar{\Qq}_{\ell}$ or $\varphi\,:\, k\lra \Cc$, we denote
$$
U_d(\varphi)=\|\varphi\|_{U_d}^{2^d}
$$
where $\|\cdot\|_{U^d}$ is the $d$-th uniformity norm, and 
$$
U_d(\sheaf{F};k)=U_d(\frfn{\sheaf{F}}{k})
$$
(in fact, we will call $U_d(\varphi)$ the Gowers $d$-pnorm of
$\varphi$ -- the `p' is silent, as in ``ptarmigan'' or ``Psmith'' --
to avoid confusion.)
\par
We will work mostly with middle-extension sheaves, in the sense
of~\cite{katz-gkm}, i.e., constructible sheaves $\sheaf{F}$ on
$\Aa^1/k$ such that, for any open set $U$ on which $\sheaf{F}$ is
lisse, with open immersion $j\,:\, U\injecte \Aa^1$, we have
$$
\sheaf{F}\simeq j_*j^*\sheaf{F}.
$$
\par
Given any constructible sheaf $\sheaf{F}$, lisse on an open set
$U\subset \Aa^1$, with $j\,:\, U\injecte \Aa^1$ the open immersion,
the direct image $j_*\sheaf{F}$ is the unique middle-extension sheaf
on $\Aa^1$ which is isomorphic to $\sheaf{F}$ on $U$. In particular
$\sheaf{F}$ and $j_*\sheaf{F}$ have the same trace functions on $U$,
but those may differ at the singularities $\Aa^1-S$. Thus the
middle-extension condition can be seen as ensuring that a lisse sheaf
on an open set $U$ of $\Aa^1$ is extended ``optimally'' to all of
$\Aa^1$.
\par
As in~\cite[\S 7]{katz-esde}, a middle-extension sheaf as above is
called pointwise pure of weight $0$, if $j^*\sheaf{F}$ is pointwise
pure of weight $0$ on $U$, and it is called arithmetically irreducible
(resp. semisimple, resp. geometrically irreducible, geometrically
semisimple) if $j^*\sheaf{F}$ corresponds to an irreducible
(resp. semisimple) representation of the fundamental group
$\pi_1(U,\bar{\eta})$ (resp. of the geometric fundamental group
$\pi_1(U\times\bar{k},\bar{\eta})$), for some geometric generic point
$\bar{\eta}$ of $U$. By the \emph{semisimplification} of $\sheaf{F}$,
we mean the middle-extension sheaf
$$
j_*\sheaf{F}^{ss}
$$
where $\sheaf{F}^{ss}$ is the semisimplification of the restriction of
$\sheaf{F}$ to $U$. 
\par
Note that
$$
\frfn{\sheaf{F}}{k}=\frfn{\sheaf{F}^{ss}}{k},
$$
so that for any question involving the trace function of $\sheaf{F}$,
we may assume that the sheaf is arithmetically semisimple.
\par
We measure the complexity of a sheaf on $\Pp^1$ over a finite field by
its \emph{conductor}: if $\sheaf{F}$ is such a sheaf, of rank
$\rank(\sheaf{F})$ with singularities at $\sing(\sheaf{F})\subset
\Pp^1$, we define the \emph{(analytic) conductor} of $\sheaf{F}$ to be
\begin{equation}\label{eq-conductor}
  \cond(\sheaf{F})=\rank(\sheaf{F})+
  \sum_{x\in \sing(\sheaf{F})}{\max(1,\swan_x(\sheaf{F}))}.
\end{equation}
\par
An important subclass of sheaves is that of \emph{tamely ramified}
sheaves, which by definition are those where $\swan_x(\sheaf{F})=0$
for all $x$, so that only the rank and number of singularities appear
as measures of complexity.
\par
If $\sheaf{F}$ is a sheaf on $U\subset \Aa^1\subset \Pp^1$, the
conductor is defined as that of the direct image to $\Pp^1$ (i.e., the
Swan conductor at any point $x\in\Pp^1$ is that of the invariants
under inertia at $x$ of the fiber of $\sheaf{F}$ over a generic
geometric point.)

We now state a first version of our main structural result (see
Theorem~\ref{th-precise} for a more precise form from which it will be
deduced; for technical reasons, we are currently only able to treat
fully the case of prime fields $k=\Zz/p\Zz$, which is the most
directly relevant to analytic number theory.)

\begin{theorem}[Algebraic structure theorem for Gowers norms]
\label{th-inverse}
Let $p$ be a prime number and let $\ell\not=p$ be an auxiliary
prime. Let $d\geq 1$ be an integer such that $p>d$.
\par
Let $\sheaf{F}$ be a middle-extension $\ell$-adic sheaf on
$\Aa^1/\Fp_p$ which is pointwise pure of weight $0$ and arithmetically
semisimple.
\par
Then we can write
$$
\frfn{\sheaf{F}}{\Fp_p}=t_1+t_2
$$
where $t_1$ and $t_2$ are themselves trace functions and:
\par
-- We have
\begin{equation}\label{eq-bound-inverse}
  U_d(t_1)\leq (5\cond(\sheaf{F}))^{(d+1)2^d}p^{-1};
\end{equation}
\par
-- There exists some non-negative integer $j\leq \rank(\sheaf{F})\leq
\cond(\sheaf{F})$, polynomials $P_i\in \Fp_p[X]$ of degree at most
$d-1$ and coefficients $\beta_i$ bounded by $\rank(\sheaf{F})$ for
$1\leq i\leq j$, such that
$$
t_2(x)=\sum_{i=1}^{j}{\beta_i e\Bigl(\frac{P_i(x)}{p}\Bigr)}.
$$
\end{theorem}


Note that the condition $p>d$ is certainly not a problem in this
horizontal direction, where $p$ is the main variable and we think of
having sheaves $\sheaf{F}_p$ for every $p$ with conductor uniformly
bounded as $p$ varies (or even growing not too fast).



The more precise structural results will imply a particularly strong
inverse theorem if $\sheaf{F}$ is assumed to be geometrically
irreducible (and also imply Theorem~\ref{th-explicit}).

\begin{corollary}[Inverse theorem for irreducible
  sheaves]\label{cor-inverse}
  Let $p$, $\ell$ and $\sheaf{F}$ be as in the theorem, and assume
  that $\sheaf{F}$ is geometrically irreducible. For $d<p$, exactly
  one of the following two possibilities holds:
\par
  -- There exists $P\in \Fp_p[X]$ of degree $\leq d-1$ and a complex
  number $\alpha$ of modulus $1$ such that
\begin{equation}\label{eq-strong-correlation}
\frtr{\sheaf{F}}{\Fp_p}{x}=\alpha e\Bigl(\frac{P(x)}{p}\Bigr)
\end{equation}
for all $x\in \Fp_p$;
\par
-- Or we have
$$
U_d(\sheaf{F};\Fp_p)\leq (5\cond(\sheaf{F}))^{(d+1)2^d}p^{-1}.
$$
\end{corollary}

\begin{remark}
 Geometric irreducibility, for sheaves $\sheaf{F}$ with small
  conductor, is equivalent with approximate $L^2$-normality of the
  trace function over $\Fp_p$, i.e., with
  the condition
$$
\frac{1}{p}\sum_{x\in\Fp_p}{
|\frtr{\sheaf{F}}{\Fp_p}{x}|^2
}\approx 1
$$
(see~\cite[Lemma 3.5]{fkm-counting} for a precise statement of
orthonormality for trace functions.)
\end{remark}

The case of tamely ramified sheaves is also simpler since
non-trivial Artin-Schreier sheaves are not tame. In fact, the
technical difficulty in extending Theorem~\ref{th-inverse} to all
finite fields, and the condition $d<p$, are then removed:

\begin{corollary}[Inverse theorem for tame sheaves]\label{cor-tame}
  Let $k$ be a finite field of characteristic $p$, let $\ell\not=p$ be
  given and let $\sheaf{F}$ be a tamely ramified $\ell$-adic
  middle-extension sheaf on $\Aa^1_k$ which does not geometrically
  contain the trivial sheaf. For $d\geq 1$, we have
$$
U_d(\sheaf{F};k)\leq (5\cond(\sheaf{F}))^{(d+1)2^d}|k|^{-1}.
$$
\end{corollary}

\begin{remark}
  Using the triangle inequality, these results of course extend to
  ``small'' linear combinations of trace functions, which include the
  (centered) characteristic functions of sets with algebraic structure
  (e.g., values of polynomials, or definable sets in the language of
  rings~\cite{cvm}). We study in~\cite{fkm-norms} some aspects of the
  norms on functions over finite fields which arise naturally from
  this point of view.
\end{remark}

A number of authors have proved inverse theorems for Gowers norms of
functions over finite fields, the most general version being the one
of Tao and Ziegler~\cite{tao-ziegler-2} (see also, without
exhaustivity, the paper~\cite{green-tao} of Green and Tao, and the
earlier paper~\cite{tao-ziegler} of Tao and Ziegler; note also that,
as shown in~\cite{green-tao}, and independently discovered by Lovett,
Meshulam and Samorodnitsky, there do exist counterexamples in large
characteristic to the most naive guess for an inverse theorem.) The
focus in these papers is different: they consider arbitrary functions
on $k^n$ as $n$ grows, and the finite field $k$ is fixed
(in~\cite{green-tao}, the functions involved are themselves
polynomials).
\par
The arguments in all these works are much more delicate than the ones
of the present paper, and this applies even more to the
article~\cite{green-tao-ziegler} of Green, Tao and Ziegler which
establishes an inverse theorem for the Gowers norms on (in effect)
$\Zz/N\Zz$, where the main variable is indeed $N\ra+\infty$. There
also, polynomial phases do not give the only obstruction to having
small Gowers norms, and more general objects related to nilmanifolds
are required.
\par
Our proof relies instead on the formalism of algebraic geometry and on
the Riemann Hypothesis over finite fields. Thus this note is another
illustration of the great power of Deligne's results, and of the
interest in dealing with trace functions of arbitrary sheaves as
objects of interest, and tools, in analytic number theory (a point of
view which is already apparent in~\cite{fkm,fkm-counting,fkm-primes}).

\begin{example}\label{ex-simple-examples}
  Possibly the simplest examples of trace functions modulo $p$ are
  given by
$$
\varphi(x)=e\Bigl(\frac{P(x)}{p}\Bigr)
$$
where $P\in\Zz[X]$ is a polynomial. In that case, the associated sheaf
is a so-called Artin-Schreier sheaf, denoted $\sheaf{L}_{\psi(P)}$ for
a suitable additive character $\psi$, and has rank $1$. It is
therefore geometrically irreducible. Its conductor is $1+\deg(P)$,
hence our theorem states that, for $p>\deg(P)$, we have
$$
\|\varphi\|_{U_d}\leq (5(1+\deg(P)))^{d+1}p^{-2^{-d}}
$$
for all $d\leq \deg(P)$, where the implied constant depends only on
$\deg(P)$ (on the other hand, it is easy to check that
$\|\varphi\|_{U_d}=1$ if $d>\deg(P)$).  The reader may check that in
this special case, our proof can be expressed using only Weil's theory
of character sums in one variable.
\par
In~\cite[Ex. 11.1.12]{tv}, Tao and Vu observe that one can prove
elementarily the  estimates 
$$
\|\varphi\|_{U_d}\leq \Bigl(\frac{d-1}{p}\Bigr)^{2^{-\deg(P)}}
$$
for $1\leq d\leq \deg(P)$, which is weaker in terms of $p$, except if
$d=\deg(P)$.
\end{example}

\subsection*{Acknowledgements}  

Thanks to R. Pink for help with questions concerning $\ell$-adic
cohomology, and to J. Wolf for interesting discussions concerning
Gowers norms. Thanks to F. Jouve for his remarks and comments
concerning the manuscript, and thanks also to B. Green for suggesting
the inclusion of a concrete statement like Theorem~\ref{th-explicit}.

\subsection*{Notation}

As usual, $|X|$ denotes the cardinality of a set, and we write
$e(z)=e^{2i\pi z}$ for any $z\in\Cc$.  We write $\Fp_p=\Zz/p\Zz$.
\par
By $f\ll g$ for $x\in X$, or $f=O(g)$ for $x\in X$, where $X$ is an
arbitrary set on which $f$ is defined, we mean synonymously that there
exists a constant $C\geq 0$ such that $|f(x)|\leq Cg(x)$ for all $x\in
X$. The ``implied constant'' refers to any value of $C$ for which this
holds. It may depend on the set $X$, which is usually specified
explicitly, or clearly determined by the context. We write $f(x)\asymp
g(x)$ to mean $f\ll g$ and $g\ll f$.
\par
For a constructible sheaf $\sheaf{F}$ on $\Aa^1/k$, and $h\in k$, we
write $[+h]^*\sheaf{F}$ for the pullback of $\sheaf{F}$ under the map
$x\mapsto x+h$. If $\sheaf{F}$ is a middle-extension sheaf on
$\Aa^1/k$, we also write $\dual(\sheaf{F})$ for the middle-extension
dual of $\sheaf{F}$, i.e., given a dense open set $j\,:\, U\injecte
\Aa^1$ where $\sheaf{F}$ is lisse, we have
$$
\dual(\sheaf{F})=j_*((j^*\sheaf{F})'),
$$
where the prime denotes the lisse sheaf of $U$ associated to the
contragredient of the representation of the fundamental group of $U$
which corresponds to $j^*\sheaf{F}$ (see~\cite[7.3.1]{katz-esde}). If
$\sheaf{F}$ is pointwise pure of weight $0$, it is known that
$$
\frtr{\dual(\sheaf{F})}{k'}{x}=\overline{\frtr{\sheaf{F}}{k'}{x}}
$$
for all finite extensions $k'/k$ and all $x\in k'$ (this property is
obvious for $x\in U(k')$ and the point is that this extends to the
singularities when the dual is suitably defined.)  Note for instance
that $\cond(\dual(\sheaf{F}))=\cond(\sheaf{F})$.

\section{Preliminaries}\label{sec-prelims}

We recall the inductive definition of the Gowers norms, as
in~\cite[Def. 11.2]{tv} (one can sheafify it for trace functions,
i.e., one can see the Gowers norm of $\frfn{\sheaf{F}}{k}$ as
essentially the sum over $x\in k$ of the trace function of a suitable
``Gowers sheaf'' $\mathcal{U}_d(\sheaf{F})$; it might in fact be
interesting to search for some kind of ``motivic'' inverse theorem for
these Gowers sheaves, but we will not pursue this point of view in
this paper.)

\begin{definition}[Gowers norms]
  Let $k$ be a finite field and $\varphi\,:\, k\lra \Cc$ an arbitrary
  function. The $U_d$-norms of $\varphi$ are defined inductively for
  $d\geq 1$ by
$$
\|\varphi\|_{U_1}^2=\frac{1}{|k|^2}\sum_{h\in k}{\sum_{x\in k}{
    \varphi(x+h)\overline{\varphi(x)} }},
$$
and
$$
\|\varphi\|_{U_{d+1}}^{2^{d+1}}=
\frac{1}{|k|}\sum_{h\in k}{
\|\xi_h(\varphi)\|_{U_d}^{2^d}}
$$
for $d\geq 1$, where
$$
\xi_h(\varphi)(x)=\varphi(x+h)\overline{\varphi(x)}.
$$
\end{definition}

We wish to apply this recursive definition to trace functions.  We
first observe the trivial bound: recalling that
$$
|\frtr{\sheaf{F}}{k}{x}|\leq \rank(\sheaf{F})
$$
for all $x\in k$ if $\sheaf{F}$ is a middle-extension sheaf on
$\Aa^1/k$ which is pointwise pure of weight $0$, we get:

\begin{corollary}\label{cor-trivial}
  Let $\sheaf{F}$ be pointwise pure of weight $0$ and rank $r\geq
  1$. Then for $d\geq 1$, we have
\begin{equation}\label{eq-trivial-pnorm}
0\leq U_d(\sheaf{F};k)\leq r^{2^d}\leq \cond(\sheaf{F})^{2^d}.
\end{equation}
\end{corollary}

Now we get a sheaf-theoretic interpretation of the recursive
definition of Gowers norms. Given a middle-extension sheaf $\sheaf{F}$
and $h\in k$, the trace function of the constructible sheaf
$$
\xi_h(\sheaf{F})=[+h]^*\sheaf{F}\otimes\dual(\sheaf{F})
$$
coincides with
$$
\xi_h(\frfn{\sheaf{F}}{k}),
$$
``almost everywhere''. Precisely it does at any $x$ which is not a
singularity of either $[+h]^*\sheaf{F}$ or $\dual(\sheaf{F})$ (this is
clear if $x$ is a singularity for neither of these, and easy to see
when $x$ is a singularity of one of them only). In fact, denoting by
$S$ the set of singularities of $\sheaf{F}$ in $\Aa^1$, we have
$$
\frfn{\xi_h(\sheaf{F})}{k}=\xi_h(\frfn{\sheaf{F}}{k})
$$
as functions on $k$ provided 
$$
h\notin E=\{h\in k^{\times}\,\mid\, S\cap (S-h)\not=\emptyset\},
$$
and moreover $\xi_h(\sheaf{F})$ is a middle-extension sheaf on $\Aa^1$
for $h\notin E$. Note that
$$
|E|\leq |S|(|S|-1)\leq \cond(\sheaf{F})(\cond(\sheaf{F})-1)
$$
since any $h\in E$ is a difference of two (distinct) elements of $S$,
and hence we get:

\begin{lemma}\label{lm-sheaf-rec}
  Let $k$ be a finite field, and let $\sheaf{F}$ be an $\ell$-adic
  middle-extension sheaf on $\Aa^1/k$ which is pointwise pure of
  weight $0$. Denote
$$
\xi_h(\sheaf{F})=[+h]^*\sheaf{F}\otimes\dual(\sheaf{F})
$$
for given $\sheaf{F}$ and $h\in k$. Then $\xi_h(\sheaf{F})$ is a
constructible sheaf, which is tame if $\sheaf{F}$ is tame.  Further,
let
$$
E=\{h\in k^{\times}\,\mid\, S\cap (S-h)\not=\emptyset\},\quad\text{
  where }\quad S=\{\text{singularities of }\sheaf{F}\}.
$$
\par
Then each $\xi_h(\sheaf{F})$ with $h\notin E$ is a middle-extension
sheaf on $\Aa^1/k$, pointwise pure of weight $0$, and we have
\begin{equation}\label{eq-induction}
  U_{d+1}(\sheaf{F};k)=
  \frac{1}{|k|}
  \sum_{h\in k-E}{U_d(\xi_h(\sheaf{F});k)}+\theta 
  \frac{|E|\rank(\sheaf{F})^{2^{d+1}}}{|k|}
\end{equation}
for any $d\geq 1$, where $|\theta|\leq 1$.
\end{lemma}

\begin{proof}
  Since it is clear that $\xi_h(\sheaf{F})$ is indeed tame if
  $\sheaf{F}$ is, the only thing that remains is to justify the error
  term in~(\ref{eq-induction}). But if $\varphi=\frfn{\sheaf{F}}{k}$,
  we have a trivial bound
$$
\frac{1}{|k|}U_d(\xi_h(\varphi))\leq
\rank(\sheaf{F})^{2\cdot 2^d}|k|^{-1}
$$
for any $h\in k$ (similar to the previous corollary), hence the
result.
\end{proof}

We now state the essential result that allows us to get optimal
bounds, which is a general version of the Riemann Hypothesis, for sums
in one variable. The version we use is as follows:

\begin{theorem}[Deligne]\label{th-rh}
  Let $p$ be a prime number and $\sheaf{F}$ an $\ell$-adic
  middle-extension sheaf on $\Aa^1/k$, pointwise pure of weight $0$,
  such that $H^2_c(\Aa^1\times\bar{k},\sheaf{F})=0$. Then we have
$$
\Bigl|
\sum_{x\in k}{\frtr{\sheaf{F}}{k}{x}}
\Bigr|\leq 2\cond(\sheaf{F})^2\sqrt{|k|}.
$$
\end{theorem}

\begin{proof}
  By the Grothendieck-Lefschetz trace formula and the Riemann
  Hypothesis, we have 
$$
\Bigl|\sum_{x\in k}{
    \frtr{\sheaf{F}}{k}{x}}\Bigr|
\leq \dim H^1_c(\Aa^1\times\bar{k},\sheaf{F})\sqrt{|k|}
$$
since
$H^0_c(\Aa^1\times\bar{k},\sheaf{F})=H^2_c(\Aa^1\times\bar{k},\sheaf{F})=0$.
\par
By the Euler-Poincar\'e formula of Grothendieck--Ogg--Shafarevich
(see, e.g.,~\cite[Ch. 14]{katz-equid}), we also know that, under our
assumptions, we have
\begin{align*}
\dim
H^1_c(\Aa^1\times\bar{k},\sheaf{F})
&=-\chi_c(\Aa^1\times\bar{k},\sheaf{F})\\
&=\sum_{x\in\sing(\sheaf{F})}{\swan_x(\sheaf{F})}
+\sum_{x\in \sing(\sheaf{F})\cap \Aa^1}{
(\rank(\sheaf{F})-\dim \sheaf{F}_x)
}
-\rank(\sheaf{F})\\
&\leq\cond(\sheaf{F})+\rank(\sheaf{F})|\sing(\sheaf{F})|
\\
&\leq 2\cond(\sheaf{F})^2
\end{align*}
hence the result.
\end{proof}

\section{Further examples}\label{sec-examples}

In this section, we will simply give a few examples of trace functions
of various kinds.  The remainder of the proof of
Theorem~\ref{th-inverse} is found in Section~\ref{sec-proof}, after a
statement of a stronger structural result in
Section~\ref{sec-refined}.

\begin{example}[Mixed characters]\label{ex-mixed}
  If $U\injecte \Aa^1$ is a dense open subset (defined over $k$), and
  $f_1$ (resp. $f_2$) is a regular function $f_1\,:\, U\lra \Aa^1$
  (resp. a non-zero regular function $f_2\,:\, U\lra \Gg_m$) both
  defined over $k$, one has the Artin-Schreier-Kummer lisse sheaf
$$
\sheaf{F}=\sheaf{L}_{\psi(f_1)}\otimes\sheaf{L}_{\chi(f_2)}
$$
defined for any non-trivial additive character $\psi\,:\, k\lra
\bar{\Qq}_{\ell}^{\times}$ and multiplicative character $\chi\,:\,
k^{\times}\lra \bar{\Qq}_{\ell}^{\times}$, which satisfy
$$
\frtr{\sheaf{F}}{k}{x}=\psi(f_1(x))\chi(f_2(x))
$$
for $x\in U(k)$. These sheaves are all of rank $1$ (in particular,
they are geometrically irreducible) and pointwise pure of weight
$0$. Moreover, possible geometric isomorphisms among them are
well-understood (see, e.g.,~\cite[Sommes Trig. (3.5.4)]{deligne}): if
$(g_1, g_2)$ is another pair of functions we have
$$
\sheaf{L}_{\psi(f_1)}\otimes\sheaf{L}_{\chi(f_2)}\simeq
\sheaf{L}_{\psi(g_1)}\otimes\sheaf{L}_{\chi(g_2)}
$$
if and only if: (1) $f_1-g_1$ is of the form 
$$
f_1-g_1=h^{|k|}-h+C
$$
for some regular function $h$ on $U$ and some constant $C\in\bar{k}$;
(2) $f_2/g_2$ is of the form
$$
\frac{f_2}{g_2}=Dh^{d}
$$
where $d\geq 2$ is the order of the multiplicative character $\chi$,
$h$ is a non-zero regular function on $U$ and $D\in\bar{k}^{\times}$. 
\par
Furthermore, the conductor of these sheaves is fairly easy to
compute. The singularities are located (at most) at $x\in
\Pp^1-U$. For each such $x$, the Swan conductor at $x$ is determined
only by $f_1$, and is bounded by the order of the pole of $f_1$ (seen
as a function $\Pp^1\lra \Pp^1$) at $x$ (there is equality if this
order is $<|k|$).
\par
In particular, if $f_1=0$ and $f_2$ is not a $d$-th power, then
$\sheaf{F}$ is tamely ramified everywhere, geometrically irreducible
and non-trivial, so that Corollary~\ref{cor-tame} applies (over
arbitrary finite fields).
\end{example}

\begin{example}[Families of Kloosterman sums]\label{ex-kloos}
  Deligne proved that, for any $p$ and $\ell\not=p$, and any
  non-trivial additive character $\psi$, there exists a
  middle-extension sheaf $\sheaf{K}\ell$ on $\Pp^1/\Fp_p$ which is
  pointwise pure of weight $0$, geometrically irreducible, lisse on
  $\Gg_m$, and satisfies
$$
\frtr{\sheaf{K}\ell}{k}{a}=-\frac{1}{\sqrt{|k|}}\sum_{x\in k^{\times}}{
\psi(ax+x^{-1})
}
$$
for any finite extension $k/\Fp_p$ and $a\in k$. This sheaf is of rank
$2$, tamely ramified at $0$ and wildly ramified at $\infty$ with Swan
conductor $1$, so that $\cond(\sheaf{K}\ell)=2+1+1=4$.
\end{example}

\begin{example}[Point-counting functions]
  The following examples are studied by Katz~\cite[Ex.
  7.10.2]{katz-esde}. Let $C/k$ be a smooth projective geometrically
  connected algebraic curve, and
$$
f\,:\, C\lra \Pp^1
$$
a non-constant map defined over $k$ of degree $d<p$. Let $D\subset C$
be the divisor of poles of $f$. Let $Z\subset C-D$ be the set of zeros
of the differential $df$, and let $S=f(Z)$ be the set of singular
values of $f$. Then, denoting by
$$
f_0\,:\, C-D\lra \Aa^1
$$
the restriction of $f$ to $C-D$, the sheaf
$$
\sheaf{F}_f=\ker(\Tr\,:\, f_{0,*}\bar{\Qq}_{\ell}\lra
\bar{\Qq}_{\ell})
$$
is a middle-extension sheaf on $\Aa^1/k$, of rank $\deg(f)-1$,
pointwise pure of weight $0$ and lisse on $\Aa^1-S$ with
$$
\frtr{\sheaf{F}_f}{k}{x}=|\{y\in C(k)\,\mid\, f(y)=x\}|-1
$$
for $x\in k-S$. This sheaf is also everywhere tamely ramified, so its
conductor is $|Z|+\deg(f)-1$. 
\par
In many cases, $\sheaf{F}_f$ is also geometrically irreducible. For
instance, this happens when $f$ is \emph{supermorse}, defined to mean
that $\deg(f)<p$, that all zeros of $df$ are simple, and that $f$
separates these zeros (i.e., $|S|=|Z|$).
\end{example}

\begin{example}[Further formalism]
  There exists a Fourier transform on middle-extension sheaves
  corresponding to the Fourier transform of trace functions, which was
  defined by Deligne and developed especially by Laumon; precisely,
  consider a middle-extension sheaf $\sheaf{F}$ which is geometrically
  irreducible, of weight $0$, and not geometrically isomorphic to
  $\sheaf{L}_{\psi}$ for some additive character $\psi$. Fix a
  non-trivial additive character $\psi$. Then the Fourier transform
  $\sheaf{G}=\ft_{\psi}(\sheaf{F})(1/2)$ satisfies
$$
\frtr{\sheaf{G}}{k}{t}=-\frac{1}{\sqrt{|k|}}
\sum_{x\in k}{\frtr{\sheaf{F}}{k}{x}\psi(tx)}
$$
for $t\in k$, and it is a middle-extension sheaf, geometrically
irreducible and pointwise pure of weight $0$ (see~\cite[\S
7]{katz-esde} for a survey and details). Moreover, one can show that
the conductor of $\sheaf{G}$ is bounded polynomially in terms of the
conductor of $\sheaf{F}$ (see, e.g.,~\cite[Prop. 7.2]{fkm}, though the
definition of conductor is slightly different there). 
\par
In particular, applying the Fourier transform to the previous
examples, we find many examples of one-parameter families of
exponential sums arising as trace functions with bounded conductor,
namely
$$
x\mapsto -\frac{1}{\sqrt{|k|}}
\sum_{y\in (C-D)(k)}{e\Bigl(\frac{x f(y)}{p}\Bigr)}
$$
for the sheaves $\sheaf{F}_f$, and
$$
x\mapsto -\frac{1}{\sqrt{|k|}}
\sum_{y\in k}{\chi(f_2(y))\psi(f_1(y)+xy)}
$$
for Artin-Schreier-Kummer sheaves (for instance, the Kloosterman sums
$\sheaf{K}\ell$ of Example~\ref{ex-kloos} can be seen as the Fourier
transform of the Artin-Schreier sheaf $\sheaf{L}_{\psi(x^{-1})}$.)
\end{example}

\section{Refined structural results}\label{sec-refined}

We will deduce Theorem~\ref{th-inverse} from the following result
which gives stronger structural information concerning the Gowers
pnorms of trace functions of middle-extension sheaves.

\begin{theorem}[Structure theorem for Gowers norms, II]
\label{th-precise}
Let $p$ be a prime number and let $\ell\not=p$ be an auxiliary
prime. Let $d\geq 1$ be an integer such that $p>d$.
\par
Let $\sheaf{F}$ be a middle-extension $\ell$-adic sheaf on
$\Aa^1/\Fp_p$ which is pointwise pure of weight $0$ and arithmetically
semisimple. Then one of the following two conditions holds:
\par
-- There exists an additive character $\psi$ of $\Fp_p$, possibly
trivial, and a polynomial $P\in \Fp_p[X]$ of degree at most $d-1$ such
that $\sheaf{F}$ geometrically contains the Artin-Schreier sheaf
$\sheaf{L}_{\psi(P)}$;
\par
-- Or else we have
\begin{equation}\label{eq-bound-inverse-repeat}
  U_d(\sheaf{F};\Fp_p)\leq (5\cond(\sheaf{F}))^{(d+1)2^d}p^{-1}.
\end{equation}
\end{theorem}

In this section, before proving this result, we check that it implies
all our previous statements.

\begin{proof}[Proof of Theorem~\ref{th-inverse}]
  Let $\sheaf{F}$ satisfy the assumptions of
  Theorem~\ref{th-inverse}. Since it is arithmetically semisimple, we
  can write
$$
\sheaf{F}=\sheaf{F}_1\oplus \sheaf{F}_2,
$$
where $\sheaf{F}_2$ is the sum of all irreducible components of
$\sheaf{F}$ which are geometrically isomorphic to an Artin-Schreier
sheaf $\sheaf{L}_{\psi(P)}$ for some polynomial $P$ of degree $\leq
d-1$. We then have
$$
\frfn{\sheaf{F}}{\Fp_p}=t_1+t_2,\text{ with }
t_i=\frfn{\sheaf{F}_i}{\Fp_p}.
$$
\par
Now we apply Theorem~\ref{th-precise} to $\sheaf{F}_1$ and
$\sheaf{F}_2$ separately. 
By construction, the first part of the dichotomy can not hold for
$\sheaf{F}_1$, and hence we get the desired estimate
$$
U_d(t_1)=U_d(\sheaf{F}_1;\Fp_p)\leq
(5\cond(\sheaf{F}))^{(d+1)2^{d}}p^{-1},
$$
by~(\ref{eq-bound-inverse-repeat}).
\par
Now for $t_2$, we write $\sheaf{F}_2$ as a direct sum of geometrically
isotypic components, which are all of the form $\sheaf{F}_{\psi(P_i)}$
for some $P_i\in \Fp_p[X]$ of degree $\leq d-1$. There are at most
$\rank(\sheaf{F})$ such components, and the trace function for each of
them is of the form
$$
x\mapsto \beta_i\psi(P_i(x)),
$$
where $\beta_i$ is the sum of the twisting factors $\alpha$ of all the
arithmetic subsheaves of $\sheaf{F}_2$ which are geometrically
isomorphic to $\sheaf{L}_{\psi(P_i)}$. Since $\sheaf{F}_2$ is
pointwise pure of weight $0$, each $\alpha$ is (under $\iota$) of
modulus $1$, and hence $|\beta_i|\leq \rank(\sheaf{F})$. Thus $t_2$ is
of the form claimed in Theorem~\ref{th-inverse}.
\end{proof}

It is also clear that Theorem~\ref{th-precise} implies
Corollary~\ref{cor-inverse}, since if $\sheaf{F}$ is geometrically
irreducible, the only possibility for the first case of the dichotomy
is that $\sheaf{F}$ be geometrically isomorphic to a sheaf
$\sheaf{L}_{\psi(P)}$ with $\deg(P)\leq d-1$, which immediately
implies~(\ref{eq-strong-correlation}). 
\par
As for Corollary~\ref{cor-tame}, it follows for $d<p$ and $k=\Fp_p$
because if $\sheaf{F}$ is tamely ramified, the only Artin-Schreier
sheaf it may geometrically contain is the trivial sheaf. The general
case of Corollary~\ref{cor-tame} follows by inspection of the
following argument (we will make remarks indicating the relevant
points).
\par
Finally, we explain how this structure result implies
Theorem~\ref{th-explicit}; this will show that many more explicit
functions with small Gowers norms can be constructed from the examples
in Section~\ref{sec-examples}.

\begin{proof}[Proof of Theorem~\ref{th-explicit}]
  We note first that because of Corollary~\ref{cor-trivial}, we can
  assume that $d<p$ for the functions $\varphi_1$, $\varphi_2$,
  $\varphi_3$ and $\varphi_4$ (the respective ranks of the sheaves
  will be $1$, $1$, $2$, $2$), the bounds being trivial for $d\geq p$.
\par
(1) The function $\varphi_1$ arises as the case of
Example~\ref{ex-mixed} for $U$ the affine line itself, $\chi$ the
Legendre character modulo $p$, $f_1=0$ and $f_2=f$. The corresponding
sheaf $\sheaf{F}_1$ has rank $1$ so is necessarily geometrically
irreducible. It is tame and not geometrically trivial if $f$ is not
proportional to the square of another polynomial, and the
singularities are $\infty$ and the zeros of $f$, so the conductor of
$\sheaf{F}_1$ is at most $2+m$. Hence the first alternative of
Theorem~\ref{th-precise} can not hold (alternatively, we can apply
Corollary~\ref{cor-inverse} here, since the sheaf is tame.)
\par
(2) The function $\varphi_2$ is the case of Example~\ref{ex-mixed} for
$U=\Aa^1-\{0\}$, $\chi=1$, $f_1(x)=x^{-1}$. The corresponding sheaf
$\sheaf{F}_2$ is of rank $1$ so geometrically irreducible. It is
tamely ramified at $\infty$ and wildly ramified with Swan conductor
$1$ at $0$, so the conductor is $3$. The classificaition of
Artin-Schreier sheaves shows that the first alternative of
Theorem~\ref{th-precise} does not hold, so we obtain the desired
bound.
\par
(3) The function $-\varphi_3$ is Example~\ref{ex-kloos}, where it is
explained that the conductor is $4$. The Kloosterman sheaf is
geometrically irreducible (for instance, because it is the
sheaf-theoretic Fourier transform of the previous sheaf $\sheaf{F}_2$,
and the Fourier transform sends geometrically irreducible sheaves to
irreducibles sheaves.) Since it is of rank $2$, the first alternative
of Theorem~\ref{th-precise} is also impossible here, and we get the
stated estimate.
\par
(4) Let
$$
\mathcal{E}\,:\, v^2=u(u-1)(u-x)
$$
denote the Legendre family of elliptic curves over $\Fp_p$, viewed as
an affine algebraic surface over $\Aa^1-\{0,1\}$ by the projection
$$
\pi\,:\, \begin{cases}
\mathcal{E}\lra \Aa^1-\{0,1\}\\
(u,v,x)\mapsto x.
\end{cases}
$$
\par
The function $\varphi_4$ arises from the middle-extension to $\Pp^1$
of the sheaf $\sheaf{F}_4=R^1\pi_!\bar{\Qq}_{\ell}(1/2)$. This is a
tame geometrically irreducible sheaf of rank $2$, ramified at
$\{0,1,\infty\}$, so the conductor is $5$ and we obtain the result as
before.
\end{proof}

\section{Proof of the inverse theorem}\label{sec-proof}

We will now prove Theorem~\ref{th-precise}. As can be expected, the
argument is by induction on $d$, and the base case $d=1$ is an easy
consequence of the Riemann Hypothesis, for any finite field:

\begin{proposition}[Inverse theorem for $d=1$]\label{pr-inverse-1}
  Let $k$ be a finite field and let $\sheaf{F}$ be a middle-extension
  sheaf on $\Aa^1/k$ which is pointwise pure of weight $0$. If
  $\sheaf{F}\otimes\bar{k}$ does not contain a trivial subsheaf,
  then we have
\begin{equation}\label{eq-d1}
  U_1(\sheaf{F};k)\leq 4\cond(\sheaf{F})^4|k|^{-1}.
\end{equation}
\end{proposition}

\begin{proof}
We have by definition
\begin{align*}
  U_1(\sheaf{F};k)&
  =\frac{1}{|k|^2} \sum_{(h,x)\in k^2}{
    \frtr{\sheaf{F}}{k}{x+h}\overline{\frtr{\sheaf{F}}{k}{x}}}\\
  &=\Bigl|\frac{1}{|k|} \sum_{x\in k}{
    \frtr{\sheaf{F}}{k}{x}}\Bigr|^2,
\end{align*}
and hence the estimate~(\ref{eq-d1}) follows immediately from
Theorem~\ref{th-rh}, unless
$$
H^2_c(\Aa^1\times\bar{k},\sheaf{F})\not=0.
$$
\par
But, if $\sheaf{F}$ is lisse on the dense open subset $U$ of $\Aa^1$,
we have
$$
H^2_c(\Aa^1\times\bar{k},\sheaf{F})=
H^2_c(U\times\bar{k},\sheaf{F})=
(\sheaf{F}_{\bar{\eta}})_{\pi_1(U\times\bar{k},\bar{\eta})}(-1)
$$
by birational invariance and the coinvariant formula for the topmost
cohomology of a lisse sheaf. Since $\sheaf{F}$ is pointwise pure of
weight $0$ on $U$, it corresponds to a representation of
$\pi_1(U\times\bar{k},\bar{\eta})$ which is geometrically semisimple
(by results of Deligne~\cite{weilii}), and therefore
$$
(\sheaf{F}_{\bar{\eta}})_{\pi_1(U\times\bar{k},\bar{\eta})}\not=0
$$
implies that $\sheaf{F}$ contains a trivial summand.
\end{proof}

We will now deal with $U_d$-pnorms, $d\geq 2$, using an induction on
$d$ based on~(\ref{eq-induction}). Precisely, consider the following
statement, for a given integer $d\geq 1$:
\par
\smallskip
\par
\textbf{Inverse($d$)}.  \textit{For any prime $p$ with $p>d$, for any
  $\ell\not=p$, for any middle extension $\ell$-adic sheaf $\sheaf{F}$
  on $\Aa^1/\Fp_p$, pointwise pure of weight $0$ and arithmetically
  semisimple, \emph{either}
  there exists an
  additive character $\psi$ of $\Fp_p$, possibly trivial, and a polynomial
  $P\in \Fp_p[X]$ of degree at most $d-1$ such that $\sheaf{F}$ contains
  geometrically a summand isomorphic to $\sheaf{L}_{\psi(P)}$,
  \emph{or else} we have
$$
U_d(\sheaf{F};\Fp_p)\leq (5\cond(\sheaf{F}))^{(d+1)2^d}p^{-1}.
$$
}
\par
\medskip
\par
Note that Proposition~\ref{pr-inverse-1} implies that
\textbf{Inverse($1$)} is valid, and that Theorem~\ref{th-precise}
simply states that \textbf{Inverse($d$)} holds for all $d\geq
1$. Hence we will be done by induction once we show:

\begin{proposition}[Induction step]\label{pr-induction}
  Let $d\geq 1$ be such that \emph{\textbf{Inverse($d$)}} holds. Then
  so does \emph{\textbf{Inverse($d+1$)}}.
\end{proposition}

For the proof of Proposition~\ref{pr-induction}, we will use two
lemmas. Before stating the first, we introduce some terminology. Given
a finite field $k$ and an open dense subset $U/k$ of $\Aa^1/k$, a
lisse sheaf $\sheaf{F}$ on $U$ is called \emph{induced} if it is
arithmetically irreducible, and the corresponding representation of
$\pi_1(U,\bar{\eta})$ is isomorphic to an induced representation
$\Ind_{H}^{\pi_1(U,\bar{\eta})}\rho_0$, for some proper normal
finite-index subgroup $H$ of $\pi_1(U,\bar{\eta})$ containing
$\pi_1(U\times\bar{k},\bar{\eta})$ and some irreducible representation
$\rho_0$. We need the following corollary of elementary representation
theory: if $\sheaf{F}$ is arithmetically irreducible on $U$, and is
not induced, then it is geometrically isotypic.


\begin{lemma}\label{lm-isotypic}
  Let $k$ be a finite field of characteristic $p$, and let $\sheaf{F}$
  be a middle-extension $\ell$-adic sheaf on $\Aa^1/k$, which is
  arithmetically irreducible and lisse on some dense open set
  $U\injecte \Aa^1$. 
\par
\emph{(1)} Either the sheaf $\sheaf{F}$ is geometrically isotypic on
$U$, or its trace function is identically zero on $U(k)$.
\par
\emph{(2)} Suppose that $\sheaf{F}$ is geometrically isotypic, and let
$\rho$ denote the geometrically irreducible representation of
$\pi_1(\bar{U},\bar{\eta})$ which corresponds to the isotypic
component of $\sheaf{F}$.  Suppose further that, for some $h\in k$,
some polynomial $P\in k[X]$ and $\ell$-adic character $\psi$, we have
a geometric summand
$$
\sheaf{L}_{\psi(P)}\injecte [+h]^*\sheaf{F}\otimes\dual(\sheaf{F})
$$
on $U$. Then we have a geometric isomorphism
$$
[+h]^*\sheaf{F}\simeq \sheaf{F}\otimes\sheaf{L}_{\psi(P)}.
$$
\end{lemma}

\begin{proof}
  (1) follows from the remark before the statement: if $\sheaf{F}$ is
  not geometrically isotypic, then it is induced so that, on $U$, the
  corresponding representation $\rho$ is given by
$$
\rho\simeq \Ind_{H}^{\pi_1(U,\bar{\eta})}\rho_0.
$$
\par
It is however elementary that, in this situation, the character of
$\rho$ is identically zero on the non-trivial cosets of $H$, and all
Frobenius $\frob_{x,k}$ corresponding to $x\in U(k)$ have this
property since we have
$$
H=\{g\in \pi_1(U,\bar{\eta})\,\mid\, \deg(g)\equiv 0\mods{m}\}
$$
for some $m\geq 2$, where $\deg$ is the degree which gives an
isomorphism
$$
\deg\,:\, \pi_1(U,\bar{\eta})/\pi_1(U\times\bar{k},\bar{\eta})\lra
\hat{\Zz},
$$
and since $\deg(\frob_{x,k})=-1$ for all $x\in U(k)$.
\par
(2) We have a geometric isomorphism $\sheaf{F}\simeq n\rho$ on $U$,
for some $n\geq 1$. Then the assumption gives
$$
\sheaf{L}_{\psi(P)}\injecte
[+h]^*\sheaf{F}\otimes\dual(\sheaf{F})\simeq
n^2 ([+h]^*\sheaf{\rho}\otimes\rho'),
$$
on $U$, and since the right-hand side is isotypic and the left-hand
side irreducible, we derive the existence of a geometric injection
$$
\sheaf{L}_{\psi(P)}\injecte [+h]^*\rho\otimes\rho',
$$
and therefore of a geometric isomorphism
$$
[+h]^*\rho\simeq \rho\otimes \sheaf{L}_{\psi(P)},
$$
and hence
$$
[+h]^*\sheaf{F}\simeq \sheaf{F}\otimes\sheaf{L}_{\psi(P)}
$$
by taking copies of this, first on $U$, and then on $\Aa^1$ because
the sheaves involved are middle-extensions.
\end{proof}

The next lemma gives some properties of lisse sheaves on
$\Aa^1_{\Fp_p}$ which are (geometrically) ``almost'' invariant under
some non-trivial translations. It complements certain results
of~\cite{fkm} (where the invariance under homographies in $\PGL_2$
acting on the projective line is a crucial issue, and where only the
base field $k=\Fp_p$ is considered.) This is also where the
restriction to $k=\Fp_p$ occurs; roughly speaking, to extend
Theorem~\ref{th-inverse} to any finite field of characteristic $p$, we
would need a similar statement as the second part of this lemma to be
valid when $G$ is an arbitrary finite subgroup of
$\bar{\Fp}_p$. However, if $\sheaf{F}$ is tame, the statement is
vacuously true (with no assumption on $d$ in (2)), simply because
there is no non-trivial tame sheaf which is lisse on $\Aa^1$.

\begin{lemma}\label{lm-peel}
  Let $\bar{k}$ be an algebraic closure of $\Fp_p$, $\ell\not=p$ an
  auxiliary prime. Let $\sheaf{F}$ be a lisse $\ell$-adic sheaf on
  $\Aa^1/\bar{k}$ such that $\sheaf{F}$ is irreducible and
  non-trivial.
\par
\emph{(1)} We have $\swan_{\infty}(\sheaf{F})\geq \rank(\sheaf{F})$,
with equality if and only if $\sheaf{F}$ is isomorphic to
$\sheaf{L}_{\psi}$ for some non-trivial $\ell$-adic additive
character.
\par
\emph{(2)} Let $d<p-1$ be given. Suppose there exists a cyclic
subgroup $G\subset \bar{k}$ of order $p$ such that we have
isomorphisms
\begin{equation}\label{eq-quasi-inv}
[+h]^*\sheaf{F}\simeq \sheaf{F}\otimes\sheaf{L}_{\psi(P_h)}
\end{equation}
on $\Aa^1$ for all $h\in G$, where $P_h\in \bar{k}[X]$ has degree
$\leq d$. Then $\sheaf{F}$ is either isomorphic to
$\sheaf{L}_{\psi(Q)}$ for some non-trivial additive character $\psi$
and polynomial $Q$ of degree $\leq d+1$, or it satisfies
$$
\swan_{\infty}(\sheaf{F})\geq p+\rank(\sheaf{F}).
$$
\end{lemma}

\begin{proof}
  (1) Since the geometric fundamental group of $\Aa^1$ is
  topologically generated by the inertia subgroups and $\sheaf{F}$ is
  lisse on $\Aa^1$, we see first that $\sheaf{F}$ is irreducible as
  representation of the inertia group $I(\infty)$ at $\infty$.
\par
Since $\sheaf{F}$ is lisse on $\Aa^1$ and non-trivial, we have
$H^0_c(\Aa^1,\sheaf{F})=H^2_c(\Aa^1,\sheaf{F})=0$, and by the
Euler-Poincar\'e formula, we get
$$
\dim H^1_c(\Aa^1,\sheaf{F})=-\chi_c(\Aa^1,\sheaf{F})
=\swan_{\infty}(\sheaf{F})-\rank(\sheaf{F}),
$$
since the Euler-Poincar\'e characteristic of $\Aa^1$ is $1$. Now, the
left-hand side is a non-negative integer, and we therefore deduce
$$
\swan_{\infty}(\sheaf{F})\geq \rank(\sheaf{F}),
$$
which is the first claim.
\par
Now suppose there is equality. Since $\sheaf{F}$ is irreducible as an
$I(\infty)$ representation, it has a unique break $\lambda$ at
$\infty$ such that
$\swan_{\infty}(\sheaf{F})=\lambda\rank(\sheaf{F})$. We therefore have
equality if and only if $\lambda=1$. 
\par
We can now apply the ``break-lowering lemma''
in~\cite[Th. 8.5.7]{katz-gkm} (it is applicable because $\sheaf{F}$ is
already $I(\infty)$-irreducible). This shows that there exists a
non-trivial additive $\ell$-adic character $\psi$ of $k$ such that
$\sheaf{G}=\sheaf{F}\otimes\sheaf{L}_{\psi(X)}$ has all breaks
$<1$. But $\sheaf{G}$ is lisse on $\Aa^1$ and still irreducible as
$I(\infty)$ representation. We claim that $\sheaf{G}$ is
(geometrically) trivial. Indeed, otherwise the inequality above would
be applicable to $\sheaf{G}$ and would give
$$
\rank(\sheaf{F})=\rank(\sheaf{G})\leq
\swan_{\infty}(\sheaf{G})<\swan_{\infty}(\sheaf{F}),
$$
which is a contradiction. Hence $\sheaf{G}$ is geometrically trivial,
and we get a geometric isomorphism $\sheaf{F}\simeq
\sheaf{L}_{\bar{\psi}}$. 
\par
(An alternative proof of the equality case goes as follows: if
$\sheaf{F}$ were not of this form, it would be a Fourier sheaf in the
sense of~\cite[\S 7.3.5]{katz-esde}; since it is lisse on $\Aa^1$ with
all breaks at $\infty$ larger than $1$, denoting by $\sheaf{G}$ its
Fourier transform, the latter would be lisse at $0$ by~\cite[Lemma
7.3.9 (3)]{katz-esde} and we would get
$$
\swan_{\infty}(\sheaf{F})=\rank(\sheaf{F})+\rank(\sheaf{G})>\rank(\sheaf{F}),
$$
by~\cite[Lemma 7.3.9, (2)]{katz-esde}, since $\sheaf{G}$ is also
irreducible by~\cite[Th. 7.3.8 (3)]{katz-esde}, hence has non-zero
rank.)
\par
(2) The finite subgroup $G\subset \bar{k}$ is cyclic, hence generated
by some $0\not=h\in \bar{k}$. Since $d<p-1$, we can find a polynomial
$Q\in \bar{k}[X]$ of degree $\leq d+1$ such that
$$
Q(X+h)-Q(X)=P_h.
$$
\par
We now form the sheaf $\sheaf{F}_1=\sheaf{F}\otimes
\sheaf{L}_{\psi(Q)}$. It is lisse on $\Aa^1$, and we have
$$
[+x]^*\sheaf{F}_1\simeq \sheaf{F}_1
$$
for any $x\in G=\Fp_ph$. Denoting 
$$
\phi\,:\, \Aa^1\lra \Aa^1/G\simeq \Aa^1
$$ 
the quotient map for the action of $G$ on $\Aa^1$, the fact that $G$
is cyclic of order $p$ implies that there exists a sheaf $\sheaf{F}_2$
on $\Aa^1/G$ such that
$$
\sheaf{F}_1\simeq \phi^*(\sheaf{F}_2).
$$
\par
We then use the invariance of Swan conductors under pushforward for
virtual representations of degree $0$ (see references in~\cite[p. 286,
line 3]{katz-mm}), i.e., the formula
$$
\swan_{\infty}(\phi^*\sheaf{F}_2-\rank(\sheaf{F}_2)\bar{\Qq}_{\ell})=
\swan_{\infty}(\phi_*(\phi^*\sheaf{F}_2-\rank(\sheaf{F}_2)\bar{\Qq}_{\ell})),
$$
where $-$ refers to the Grothendieck ring of lisse sheaves on
$\Aa^1$. The left-hand side is equal to
$$
\swan_{\infty}(\phi^*\sheaf{F}_2)=\swan_{\infty}(\sheaf{F}_1),
$$
while the right-hand side is equal to
\begin{align}
  \swan_{\infty}(\phi_*(\phi^*\sheaf{F}_2-\rank(\sheaf{F}_2)\bar{\Qq}_{\ell}))
  &= \sum_{\eta\in
    \hat{G}}{(\swan_{\infty}(\sheaf{F}_2\otimes\sheaf{L}_{\eta})
    -\rank(\sheaf{F}_2)\swan_{\infty}(\sheaf{L}_{\eta})
    )}\nonumber\\
  &= \sum_{\eta\in
    \hat{G}}{(\swan_{\infty}(\sheaf{F}_2\otimes\sheaf{L}_{\eta})
    -\rank(\sheaf{F}_2))}+\rank(\sheaf{F}_2),
\label{eq-induc-restr}
\end{align}
where $\eta$ runs over $\ell$-adic characters of $G$, and the
$\sheaf{L}_{\eta}$ are the corresponding lisse sheaves on
$\Aa^1$.
\par
Since $\sheaf{F}$ is irreducible, so is $\sheaf{F}_2$, and the twists
$\sheaf{F}_2\otimes\sheaf{L}_{\eta}$ in the sum. If one term in this
sum is zero, we get
$$
\swan_{\infty}(\sheaf{F}_2\otimes
\sheaf{L}_{\eta})=\rank(\sheaf{F}_2\otimes\sheaf{L}_{\eta}) 
$$
and therefore, by the equality case of (1), we have
$$
\sheaf{F}_2\simeq \sheaf{L}_{\bar{\eta}}\otimes \sheaf{L}_{\psi'}
$$
for some additive character $\psi'$. Pulling back under $\phi$, it
follows that $\sheaf{F}_1$ is also an Artin-Scheier sheaf
$\sheaf{L}_{\psi(aX)}$ for some $a$, and hence
$$
\sheaf{F}\simeq \sheaf{L}_{\psi(aX)}\otimes\sheaf{L}_{\psi(-Q)}
$$
in that case. 
\par
On the other hand, if none of the terms in the sum vanishes, we get
$$
\swan_{\infty}(\sheaf{F}_1)=\swan_{\infty}(\phi^*\sheaf{F}_2)
\geq p+\rank(\sheaf{F}_1).
$$
\par
In particular, by assumption, this is $>d$, and hence
$$
\swan_{\infty}(\sheaf{F})=
\swan_{\infty}(\sheaf{F}_1\otimes\sheaf{L}_{\psi(-Q)})=
\swan_{\infty}(\sheaf{F}_1) \geq p+\rank(\sheaf{F}).
$$

\end{proof}

The final lemma gives an upper-bound for the conductor of
$\xi_h(\sheaf{F})$.

\begin{lemma}\label{lm-cond-xih}
  Let $\sheaf{F}$ be a middle-extension sheaf on $\Aa^1/\bar{\Fp}_p$
  and $h\in \bar{\Fp}_p$ such that the set of singularities of
  $\sheaf{F}$ and $[+h]^*\sheaf{F}$ in $\Aa^1$ are distinct. Then the
  conductor of $\xi_h(\sheaf{F})$ satisfies
$$
\cond(\xi_h(\sheaf{F}))\leq 5\cond(\sheaf{F})^2.
$$
\end{lemma}

\begin{proof}
Indeed, since the singularities are disjoint, we have
$$
\cond(\xi_h(\sheaf{F})\leq
\rank(\sheaf{F})^2+2\rank(\sheaf{F})\sum_{x\in
  S}{\swan_x(\sheaf{F})} +\swan_{\infty}(\xi_h(\sheaf{F})).
$$
\par
But from known properties of Swan
conductors~\cite[(3.2)]{esnault-kerz}, we have
$$
\swan_{\infty}(\xi_h(\sheaf{F})) \leq
\rank(\sheaf{F})\swan_{\infty}([+h]^*\sheaf{F})+
\rank(\sheaf{F})\swan_{\infty}(\dual(\sheaf{F})) \leq
2\cond(\sheaf{F})^2,
$$
hence the result.
\end{proof}

We are now able to conclude the inductive proof of the inverse
theorem. The reader is encouraged to check the tame case, for an
arbitrary finite field and with no assumption on $d$ compared with the
characteristic $p$.

\begin{proof}[Proof of Proposition~\ref{pr-induction}]
  We start with the data for a case of \textbf{Inverse($d+1$)}: $p$ is
  a prime number $>d+1$, $\sheaf{F}$ is a middle-extension sheaf of
  weight $0$ on $\Aa^1/\Fp_p$ which is pointwise pure of weight $0$
  and arithmetically semisimple.
  We will show that one of the two conditions in
  \textbf{Inverse($d+1$)} holds. For notational simplicity, we write
  $c=\cond(\sheaf{F})$ and $S=\sing(\sheaf{F})\cap \Aa^1$.
\par
Let $U/\Fp_p$ be the complement of the singularities $S$ of
$\sheaf{F}$ in $\Aa^1/\Fp_p$, so that $\sheaf{F}$ is lisse on
$U/\Fp_p$. Let
$$
\sheaf{F}=\bigoplus_{1\leq i\leq r}{\sheaf{F}_i}
$$
be a decomposition of $\sheaf{F}$ into direct sum of arithmetically
irreducible middle-extension sheaves. Note that $r\leq
\rank(\sheaf{F})\leq c$ and each $\sheaf{F}_i$ also has conductor
$\leq c$, and hence we have
\begin{equation}\label{eq-bound-1}
  U_{d+1}(\sheaf{F};\Fp_p)\leq c^{2^{d+1}}\sum_{i=1}^r{U_{d+1}(\sheaf{F}_i;\Fp_p)}.
\end{equation}
\par
We now consider a fixed $i$, and the arithmetically irreducible sheaf
$\sheaf{F}_i$. If $\sheaf{F}_i$ is induced, its trace function is zero
on $U$, and is bounded by $\rank(\sheaf{F}_i)\leq c$ on the
complement, which contains at most $c$ points, so that a trivial
estimate gives
\begin{equation}\label{eq-induced}
U_{d+1}(\sheaf{F}_i;\Fp_p)\leq \frac{c^{1+2^{d+1}}}{p}.
\end{equation}
\par
Now we assume that $\sheaf{F}_i$ is not induced. By
Lemma~\ref{lm-sheaf-rec}, noting that the singularities of
$\sheaf{F}_i$ are among those of $\sheaf{F}$, we have the inductive
formula
\begin{equation}\label{eq-inductive}
U_{d+1}(\sheaf{F}_i;\Fp_p)= \frac{1}{p} \sum_{h\in
  \Fp_p-E}{U_d(\xi_h(\sheaf{F}_i);\Fp_p)}+\theta
\frac{c^{2+2^{d+1}}}{p}
\end{equation}
with $|\theta|\leq 1$, where
$$
E=\{h\in \Fp_p^{\times}\,\mid\, S\cap (S-h)\not=\emptyset\}.
$$
\par
Each term in the sum can be trivially bounded by
\begin{equation}\label{eq-bound-2}
  \frac{1}{p}U_d(\xi_h(\sheaf{F}_i);\Fp_p)\leq
  \rank(\xi_h(\sheaf{F}_i))^{2^d}p^{-1}
  =\rank(\sheaf{F}_i)^{2\cdot 2^d}p^{-1}
  \leq c^{2^{d+1}}p^{-1}
\end{equation}
(which we can therefore use for some exceptional $h$, provided their
number is not too large in terms of $c$).
\par
Furthermore, we know that for each $h\in \Fp_p^{\times}-E$, the sheaf
$\xi_h(\sheaf{F}_i)$ is a middle-extension sheaf, lisse on $U_h=U\cap
(U-h)$ and pointwise pure of weight $0$. By Lemma~\ref{lm-cond-xih},
its conductor is $\leq 5c^2$.
\par
We can therefore apply the induction assumption
\textbf{Inverse($d$)}. We obtain the bound
\begin{equation}\label{eq-apply-induction}
U_{d}(\xi_h(\sheaf{F}_i);\Fp_p)\leq
(5\cond(\xi_h\sheaf{F}_i))^{(d+1)2^d}p^{-1} \leq (5c)^{(d+1)2^{d+1}}p^{-1},
\end{equation}
for all those $h\in \Fp_p^{\times}-E$ such that there does \emph{not}
exist some $P_h\in \Fp_p[X]$ with $\deg(P_h)\leq d-1$ with a geometric
embedding
$$
\sheaf{L}_{\psi(P)}\injecte
\xi_h(\sheaf{F}_i)=[+h]^*\sheaf{F}_i\otimes \dual(\sheaf{F}_i).
$$
\par
We denote by $F_i\subset \Fp_p-E$ the set of exceptional $h$ for which
this last property holds (including $h=0$). By
Lemma~\ref{lm-isotypic}, (2), if $h\in F_i$, we have a geometric
isomorphism
$$
[+h]^*\sheaf{F}_i\simeq \sheaf{F}_i\otimes \sheaf{L}_{\psi(P_h)}
$$
for some polynomial $P_h$ of degree $\leq d-1$, and hence
\begin{multline*}
  F_i\subset G=\{h\in \Fp_p\,\mid\, [+h]^*\sheaf{F}_i\text{ is
    geometrically isomorphic to } \\ \sheaf{F}_i\otimes
  \sheaf{L}_{\psi(P)}\text{ for some $P$ of degree $\leq d-1$}\}.
\end{multline*}
\par
This subset $G$ is an additive subgroup of $\Fp_p$, hence either
trivial or equal to $\Fp_p$. In the former case, we are
done. Otherwise, we first note that if $\sheaf{F}_i$ has a singularity
$a\in \Aa^1$, all elements in its orbit under the action of $G$ are
also singularities, i.e., $|G|\leq c$. We can apply~(\ref{eq-bound-2})
for all $h\in G$, getting a contribution
\begin{equation}\label{eq-bound-3}
\leq |G|c^{2^{d+1}}p^{-1}\leq c^{1+2^{d+1}}p^{-1}
\end{equation}
for these terms.
\par
The other possibility is that $\sheaf{F}_i$ is lisse on $\Aa^1$. We
can then apply Lemma~\ref{lm-peel}, (2) (to the geometrically
irreducible component of the arithmetically irreducible but
non-induced sheaf $\sheaf{F}_i$)) and two possibilites arise: either
$\sheaf{F}_i$ is geometrically isomorphic to a direct sum of copies of
$\sheaf{L}_{\psi(Q)}$ for some polynomial $Q$ of degree $\leq d$, or
otherwise we have
$$
c\geq \swan_{\infty}(\sheaf{F}_i)\geq |G|=p,
$$
in which case we also get the bound~(\ref{eq-bound-3}) for this
contribution.
\par
Combining~(\ref{eq-bound-1}),~(\ref{eq-induced}),~(\ref{eq-bound-3})
and the average of the inductive bounds~(\ref{eq-apply-induction}), we
get
$$
  U_{d+1}(\sheaf{F};\Fp_p)\leq Ap^{-1},
$$
where
$$
A=c^{2^{d+1}}\left\{
c^{1+2^{d+1}}+c^{2+2^{d+1}}+(5c)^{(d+1)2^{d+1}}
\right\},
$$
and in order to finish the induction, we must check that $A\leq
(5c)^{(d+2)2^{d+1}}$, for $d\geq 1$, which is easily done, e.g., using
the bound
$$
c^{2^{d+1}}\times \Bigl\{c^{1+2^{d+1}}+c^{2+2^{d+1}}
\Bigr\} \leq 2c^{(d+2)2^{d+1}}.
$$
\end{proof}

\end{document}